\documentclass[reqno, a4paper, 11pt]{amsart}
\usepackage{amssymb,amsmath,mathtools,amsthm}
\usepackage{enumerate}
\usepackage{xcolor}
\usepackage{mathrsfs}
\usepackage{enumitem}
\usepackage{hyperref}
\usepackage{cleveref}

\makeatletter
\newcommand{\mylabel}[2]{#2\def\@currentlabel{#2}\label{#1}}
\makeatother

\newcommand{\W}{\mathcal W}

\newcommand{\R}{\mathbb R}

\newcommand{\Pc}{\mathcal P}

\def\E{{\mathbb E}}

\def\R{{\mathbb R}}

\usepackage{color}

\usepackage{hyperref}
\usepackage{txfonts}

\usepackage{newtxtext}

\usepackage{tikz}
\usetikzlibrary{external}
\usetikzlibrary{backgrounds,patterns,decorations.pathreplacing,decorations.pathmorphing}
\usepackage{caption}
\usepackage{accents}

\newtheorem{theorem}{Theorem}

\theoremstyle{remark}

\begin{document}

\begin{abstract}
	We provide a short proof of the intriguing characterisation of the convex order given by Wiesel and Zhang \cite{WiZh22}.\\[0.1cm]
\end{abstract}

\author{Beatrice~Acciaio}
\author{Gudmund~Pammer}
\thanks{ETH Zurich, Switzerland,
\href{mailto:beatrice.acciaio@math.ethz.ch}{beatrice.acciaio@math.ethz.ch}, \href{mailto:gudmund.pammer@math.ethz.ch}{gudmund.pammer@math.ethz.ch}}
\title{A short proof of the characterisation of convex order using the 2-Wasserstein distance}

\maketitle

\section{Statement and proof}
We consider the set $\mathcal P_2(\R^d)$ of probabilities with finite second moment on the Euclidean space $(\R^d, |\cdot|)$.
The Wasserstein-2 distance $\mathcal W_2$ on $\mathcal P_2(\R^d)$ is given by
\begin{equation}
	\label{eq:W2}
	\textstyle
	\mathcal W_2^2(\mu,\nu) := \inf_{\pi \in \Pi(\mu,\nu)} \int |x - y|^2 \, \pi(dx,dy),
\end{equation}
where $\Pi(\mu,\nu)$ denotes the set of all probabilities on $\R^d \times \R^d$ with marginals $\mu$ and $\nu$.
Recall that the infimum in \eqref{eq:W2} is attained, see \cite[Theorem~4.1]{Vi09}, and write $\Pi^\ast(\mu,\nu)$ to denote the corresponding set of optimizers.

\begin{theorem}
	Let $\mu, \nu \in \mathcal P_2(\R^d)$.
	Then the following are equivalent:
	\begin{enumerate}[label=(\roman*)]
		\item \label{it:thm_1}
		$\mu \le_c \nu$, that is, $\int f(x) \, \mu(dx) \le \int f(x) \, \nu(dx)$\; for all convex $f \colon \R^d \to \R$;
		\item \label{it:thm_2}
		$\W_2^2(\nu,\rho) - \W_2^2(\mu,\rho) \le \int |x|^2 \, \nu(dx) - \int |x|^2 \, \mu(dx)$\; for all $\rho \in \Pc_2(\R^d)$.
	\end{enumerate}
\end{theorem}

\begin{proof}
	The implication \ref{it:thm_1}$\implies$\ref{it:thm_2} was already shown in \cite[Equation 2.2]{AlJo20}:
	Fix $\rho \in \Pc_2(\R^d)$, let $\pi^\ast \in \Pi^\ast(\mu,\rho)$, and let $\pi^M$ be a martingale coupling between $\mu$ and $\nu$ (which exists by \cite[Theorem~8]{St65}).
	By virtue of the conditionally independent product (see \cite[Definition~2.8]{Ed19}), there exists a probability space with random variables $X,Y,Z$ such that $(X,Z) \sim \pi^\ast$, and $(X,Y) \sim \pi^M$ is a one-step martingale in the filtration $(\sigma(X,Z), \sigma(X,Y,Z))$. Then
	\begin{align*}
		\W_2^2(\nu,\rho) &\le \E[|Y-Z|^2] = \E[|X - Y|^2 + |X-Z|^2]
		\\
		&= \W_2^2(\mu,\rho) + \E[|Y|^2 - |X|^2] = \W_2^2(\mu,\rho) + 
		\textstyle
		\int |y|^2 \, \nu(dy) - \int |x|^2 \, \mu(dx).
	\end{align*}

	To see the reverse direction, \ref{it:thm_2}$\implies$\ref{it:thm_1}, first observe that it suffices to test the convex ordering with convex functions $f\colon \R^d \to \R$ with at most linear growth.
	Denote by $\nabla f(x)$ a measurable selection of $\partial f(x)$, the subdifferential of $f$ at $x$, and set $\rho := (\nabla f)_\# \mu$.
	Next, let $\pi^\ast \in \Pi^\ast(\rho,\nu)$ and consider a probability space with random variables $(X,Y)$ such that $X \sim \mu$ and $(\nabla f(X),Y) \sim \pi^\ast$ (which again exists by the same argument as above).
	By Brenier's theorem, optimality of $\pi^\ast$, and by definition of the subdifferential, it follows from \ref{it:thm_2} that
	\begin{align*}
		0&\le \tfrac12 \left( \W_2^2(\mu,\rho) - \W_2^2(\nu,\rho) - \E[|X|^2 - |Y|^2] \right)
		\\
		&=
		\tfrac12\E[|X - \nabla f(X)|^2 - |Y - \nabla f(X)|^2 - |X|^2 + |Y|^2]
		\\
		&= \E[\nabla f(X) \cdot (Y - X)] 
		\textstyle
		\le \E[f(Y) - f(X)] = \int f(y) \, \nu(dy) - \int f(x) \, \mu(dx).\qedhere
	\end{align*}
\end{proof}


\bibliographystyle{abbrv}
\bibliography{../MBjointbib/joint_biblio}
\end{document}